%% file: main.tex
\documentclass[11pt]{article}
\usepackage[margin=1in]{geometry}
\usepackage{libertine}

\usepackage{datetime}
\usepackage{multirow}
\usepackage{booktabs}
\usepackage{enumitem}
\usepackage{wrapfig}
\usepackage{subcaption}
\usepackage[T1]{fontenc}
\usepackage[utf8]{inputenc}
\usepackage{pgfplots}
\usepackage{amsmath, amssymb}
\usepackage{amsthm}
\usepackage{color}
\usepackage{cases}
\usepackage{xparse}
\usepackage{xargs}
\usepackage{appendix}
\usepackage[boxed,commentsnumbered,noend]{algorithm2e}
\usepackage{algpseudocode}
\usepackage{verbatim, xspace}
\usepackage{tikz}
\usepackage{mathtools}

\usepackage{hyperref}
\usepackage[capitalize,noabbrev,nameinlink]{cleveref}
\usepackage{pgfplots}

\usepackage{xcolor}

\newcommand{\citet}[1]{\cite{#1}}
\SetAlgorithmName{Isolation scheme}{isolation scheme}{List of isolaiton schemes}
\usepackage[colorinlistoftodos,prependcaption,textsize=tiny]{todonotes}
\newcommandx{\unsure}[2][1=]{\todo[linecolor=green,backgroundcolor=green!25,bordercolor=green,#1]{\normalsize #2}}
\newcommandx{\improvement}[2][1=]{\todo[inline,linecolor=blue,backgroundcolor=blue!05,bordercolor=blue,#1]{\normalsize #2}}
\newcommandx{\info}[2][1=]{\todo[linecolor=yellow,backgroundcolor=yellow!25,bordercolor=yellow,#1]{#2}}
\newcommandx{\floatmodel}[2][1=]{\todo[inline,linecolor=red,backgroundcolor=yellow!25,bordercolor=yellow,#1]{#2}}
\newcommandx{\thiswillnotshow}[2][1=]{\todo[disable,#1]{#2}}
\newcommandx{\karol}[2][1=]{\todo[inline,linecolor=blue,backgroundcolor=blue!25,bordercolor=blue,caption={\normalsize \textbf{Karol}},#1]{\normalsize #2}}
\newcommandx{\jana}[2][1=]{\todo[inline,linecolor=red,backgroundcolor=red!25,bordercolor=red,caption={\normalsize \textbf{Jana}},#1]{\normalsize #2}}
\newcommandx{\michal}[2][1=]{\todo[inline,linecolor=gray,backgroundcolor=red!25,bordercolor=red,caption={\normalsize \textbf{Micha\l{}}},#1]{\normalsize #2}}
\newcommandx{\marco}[2][1=]{\todo[inline,linecolor=green,backgroundcolor=green!25,bordercolor=green,caption={\normalsize \textbf{Marco}},#1]{\normalsize #2}}

\crefformat{page}{#2page~#1#3}%
\Crefformat{page}{#2Page~#1#3}%
\crefformat{equation}{#2(#1)#3}%
\Crefformat{equation}{#2(#1)#3}%
\crefformat{enumi}{#2(#1)#3}
\Crefformat{enumi}{#2(#1)#3}

\newtheorem{theorem}{Theorem}
\newtheorem{lemma}{Lemma}
\newtheorem{conjecture}{Conjecture}

\newtheorem{corollary}{Corollary}

\newcommand{\eps}{\varepsilon}

\newcommand{\Oh}{\mathcal{O}}

\newcommand{\nat}{\mathbb{N}}

\newcommand{\Ss}{S}

\newcommand{\Mm}{\mathcal{M}}

\newcommand{\Gg}{\mathcal{G}}

\newcommand{\Bb}{\mathcal{B}}
\newcommand{\Ii}{\mathcal{I}}

\newcommand{\graphs}{\Gg_\mathrm{seg}}

\renewcommand{\leq}{\leqslant}
\renewcommand{\geq}{\geqslant}
\renewcommand{\le}{\leqslant}
\renewcommand{\ge}{\geqslant}

\title{Independence number of intersection graphs of axis-parallel segments}
\date{}

\author{
	Marco Caoduro\footnote{Laboratoire G-SCOP, Univ. Grenoble Alpes, France, \textsf{marco.caoduro@grenoble-inp.fr}.}
	\and
    Jana Cslovjecsek\footnote{EPFL, Switzerland, \textsf{jana.cslovjecsek@epfl.ch}.}
    \and
    Micha\l{} Pilipczuk\footnote{Institute of Informatics, University of
    Warsaw, Poland, \textsf{michal.pilipczuk@mimuw.edu.pl}. This work is a part of
    the project BOBR that has received funding from the European
    Research Council (ERC) under the European Union's Horizon 2020 research and
    innovation programme (grant agreement No 948057).}
    \and
    Karol W\k{e}grzycki\footnote{Saarland University and Max Planck Institute for Informatics,
        Saarbr\"ucken, Germany, \textsf{wegrzycki@cs.uni-saarland.de}. 
    This work is part of the project TIPEA that has
    received funding from the European Research Council (ERC) under the European Union's Horizon
    2020 research and innovation programme (grant agreement No 850979).}
}

\begin{document}

\maketitle

\thispagestyle{empty}
\input{chapters/abstract}

\begin{picture}(0,0)
\put(462,-370)
{\hbox{\includegraphics[width=40px]{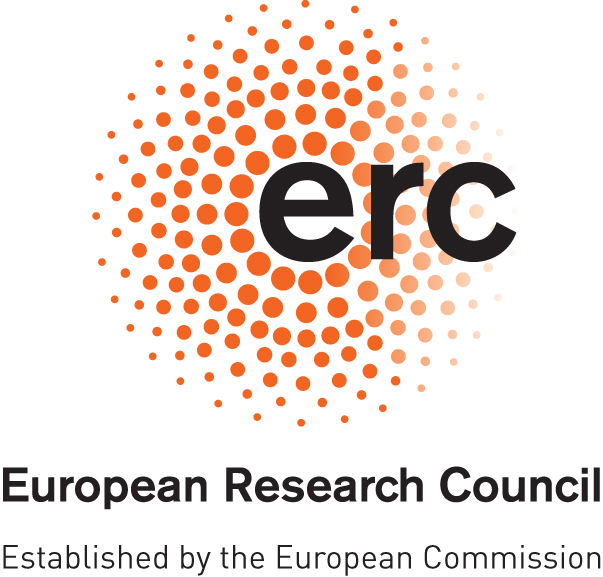}}}
\put(452,-430)
{\hbox{\includegraphics[width=60px]{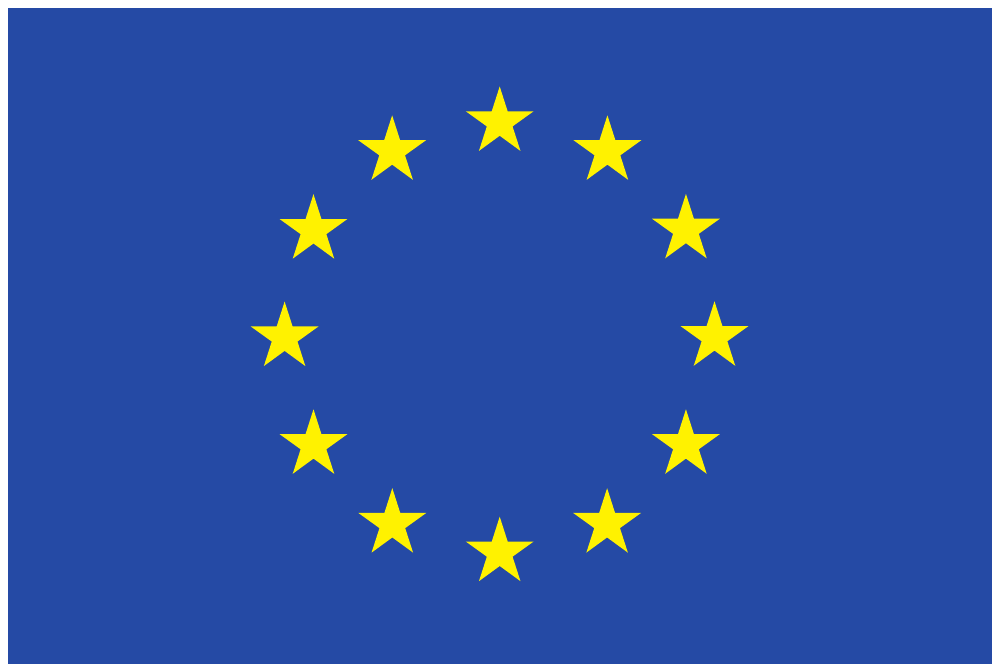}}}
\end{picture}

\clearpage
\setcounter{page}{1}

\input{chapters/intro2}
\input{chapters/algorithm}
\input{chapters/section1}
\input{chapters/acknowledgement}

\bibliographystyle{abbrv}
\bibliography{bib}



\end{document}

%% file: chapters/abstract.tex

\begin{abstract}
We prove that for any triangle-free intersection
graph of $n$ axis-parallel segments in the plane, the independence number $\alpha$ of this graph is at least $\alpha \ge n/4 +
\Omega(\sqrt{n})$. We complement this with a construction of a graph in this class
satisfying $\alpha \le n/4 + c \sqrt{n}$ for an absolute constant $c$, which demonstrates
the optimality of our~result.
\end{abstract}

%% file: chapters/intro2.tex
\section{Introduction}

For a graph $G$, the {\em{independence number}} $\alpha(G)$ is the maximum size of an 
independent set in $G$. 
Both lower and upper bounds on the independence number were intensively studied in various graph classes.
In this paper, we study the independence number in classes of geometric intersection graphs.

For a family of geometric objects $\Ss$ in the plane, the intersection graph $G(\Ss)$ has vertex set $\Ss$
and two objects are considered adjacent if they intersect. Naturally,
the independence number $\alpha(\Ss)$ is defined as the maximum size of a subset of objects that are pairwise disjoint. 

A simple lower bound on the independence number can be often obtained by studying the \emph{chromatic number} $\chi(G)$ --- the minimum number of colors needed to properly color the vertices
of $G$ --- and using the obvious inequality $\alpha(G) \ge n/\chi(G)$.
This strategy does not alway provide optimum lower bounds, which will be also the case in this work.

Specifically, we
consider intersection graphs of axis-parallel segments in the plane
where no three segments intersect at a single point. For simplicity, we will
denote this class of graphs by $\graphs$. Observe that we have $\chi(G)\le 4$ for every $G \in \graphs$, because we can use two colors to properly color the horizontal segments, and another two for the vertical segments. Hence, if $G\in \graphs$ has $n$ vertices, then $\alpha(G)\ge n/4$. 

Our two main results, presented below, prove that this simple lower bound can be always improved by an additive term of the order $\sqrt{n}$, but no further improvement is possible.

\begin{theorem}\label{thm:lower bound}
    Let $G$ be a graph in $\graphs$ with $n$ vertices. Then the independence
    number of $G$ is at least
    \begin{displaymath}
        \alpha(G) \ge  \frac{n}{4} + c_1\sqrt{n},
    \end{displaymath}
    for some absolute constant $c_1$.
\end{theorem}

\begin{theorem}
	\label{thm:upper_bound}
    For any $n \in \nat$ there exists a graph $G$ in $\graphs$ on $n$ vertices with independence number
    \begin{displaymath}
        \alpha(G) \le \frac{n}{4} + c_2 \sqrt{n},
    \end{displaymath}
    for some absolute constant $c_2$.
\end{theorem}

\paragraph*{Consequences.}
The independence number is often studied in
relation to the clique covering number. A \emph{clique} in a graph is a set of pairwise adjacent vertices, and the \emph{clique
covering number} $\theta(G)$ of a graph $G$ is defined as the minimal size of a
partition of the vertex set of $G$ into cliques. For any graph $G$, the clique
covering number $\theta(G)$ is a natural upper bound on the independence number
$\alpha(G)$. Indeed, an independent set contains at most one vertex from each
clique. This implies that for any graph $G$ the ratio $\theta(G)/\alpha(G)$ is
at least one. Giving an upper bound on this ratio is a question that was studied for 
several classes of intersection graphs (\textit{e}.\textit{g}.\ \cite{1985_Gyarfas}, \cite{2008_Kim}).
In this topic, the main open question concerns the relation
between the independence number and the clique covering number in intersection graphs of axis-parallel rectangles.

\begin{conjecture}[Wegner \cite{1965_Wegner}, 1965] \label{conj:wegner}
	Let $G$ be the intersection graph of a set of axis-parallel rectangles in the plane. Then
	\[\theta(G) \leq 2\alpha(G) - 1.\]
\end{conjecture}

For an intersection graph $G$ of axis-parallel rectangles the best
known bound on the clique covering number is $\theta(G) =
\Oh(\alpha(G)\log^2(\log(\alpha(G))))$ by Correa et. al.~\cite{2015_Correa}. In
particular, no linear upper bounds are known.  Even, obtaining lower bounds on the
maximal ratio $\theta/\alpha$ was until recently an elusive task.  For nearly
thirty years after Wegner formulated his conjecture, the largest known ratio
remained $3/2$, obtained by taking five axis-parallel rectangles forming a
cycle. In 1993, Fon-Der-Flaass and Kostochka presented a family of axis-parallel
rectangles with clique cover number 5 and independence number
3~\cite{1993_FonDerFlaass}. Only in 2015, Jel\'inek constructed families of
rectangles with ratio $\theta/\alpha$ arbitrarily close to 2, showing that the
constant of $2$ in Wegner's conjecture cannot be improved\footnote{The former construction is attributed to Jel\'inek in \cite[Ackowledgment]{2015_Correa}. }.

One consequence of our results is a proof that the ratio $2$ in Wegner's
conjecture cannot be improved even in the highly restricted case of axis-parallel
segments, even with the assumption of triangle-freeness.  More precisely, we have the following corollary.

\begin{corollary}
    \label{cor:ratio}
	For any $\eps>0$, there exists a graph $G$ in $\graphs$ such that
    \[\theta(G)\ge (2-\eps)\alpha(G).\]
\end{corollary}

\Cref{cor:ratio} is a consequence of the full version of
\cref{thm:upper_bound} (see Section~\ref{sec:upper_bound}).



\Cref{cor:ratio} can be further strengthened to the fractional setting, implying a lower bound on the integrality gap of the standard LP relaxation of the independent set problem. Namely, consider the {\em{fractional independence number}} of a graph $G$, denoted $\alpha^\star(G)$, which is defined similarly to $\alpha(G)$, but every vertex $u$ can be included in the solution with a fractional multiplicity $x_u\in [0,1]$, and the constraints are that $x_u+x_v\leq 1$ for every edge $uv$ of $G$.  Similarly, in the {\em{fractional clique cover number}} $\theta^\star(G)$ every clique $K$ in $G$ can be included in the cover with a fractional multiplicity $y_K\in [0,1]$, and the constraints are that $\sum_{K\colon v\in K} y_K\geq 1$ for every vertex $v$. In triangle-free graphs the linear programs defining $\alpha^\star(G)$ and $\theta^\star(G)$ are dual to each other, hence
$$\alpha(G)\leq \alpha^\star(G) = \theta^\star(G)\leq \theta(G)\qquad\textrm{for every triangle-free }G.$$
The proof of \Cref{cor:ratio}, based on the full version of \cref{thm:upper_bound}, actually gives the following.

\begin{corollary}\label{cor:LP}
 For any $\eps>0$, there exists a graph $G$ in $\graphs$ such that
    \[\alpha^\star(G)\ge (2-\eps)\alpha(G).\]
 Consequently, the integrality gap of the standard LP relaxation of the maximum independent set problem in graphs from $\graphs$ is not smaller than $2$.
\end{corollary}

We note that recently, G\'alvez et al. gave a polynomial-time $(2+\eps)$-approximation algorithm for the maximum independent set problem in intersection graphs of axis-parallel rectangles~\cite{GalvezKMMPW22-arxiv,GalvezKMMPW22}. Thus, \cref{cor:LP} shows that one cannot improve upon the approximation ratio of $2$ by only relying on the standard LP relaxation, even in the case of axis-parallel segments. Note that in this case, obtaining a $2$-approximation algorithm is very easy: restricting attention to either horizontal or vertical segments reduces the problem to the setting of interval graphs, where it is polynomial-time solvable.
%

%% file: chapters/algorithm.tex
\section{The lower bound: proof of \cref{thm:lower bound}}
\label{sec:lower_bound}
\newcommand{\ply}{\mathrm{ply}}
\newcommand{\lhor}{\ell_{\mathrm{horizontal}}}
\newcommand{\lver}{\ell_{\mathrm{vertical}}}
\newcommand{\leven}{\ell_{\mathrm{even}}}
\newcommand{\lodd}{\ell_{\mathrm{odd}}}
\newcommand{\seven}{s_{\mathrm{even}}}
\newcommand{\sodd}{s_{\mathrm{odd}}}
\newcommand{\hor}{\mathrm{horizontal}}
\newcommand{\ver}{\mathrm{vertical}}

The goal of this section is to prove \cref{thm:lower bound}.  For this, we examine a graph $G\in \graphs$, and we
exhibit three different independent sets in $G$ by
constructing three different subsets of disjoint segments.  A trade-off between
these three independent sets then results in a lower bound.

The set of geometric objects $S$ is called a \emph{representation} of its
intersection graph $G(S)$ (note that a graph can have multiple representations).
Our proof starts with some observations on the possible sets of segments representing a graph in the class $\graphs$.
Let $G$ be a graph in $\graphs$ with $n$ vertices and let $\Ss$ be a representation of $G$. Thus, $\Ss$ consists of axis-parallel segments, no three of which meet at one point.
We may assume that in $\Ss$ every two parallel segments that intersect meet at a single point, called the \emph{meeting point}.
If two segments do not meet at a single point, we can choose any common point and shorten both segments up to this common point.
Since no three segments of $\Ss$ meet at one point, all intersections are preserved and the modified set of segments is still a representation of $G$.
Further, we may assume that if two orthogonal segments intersect, their intersection point lies in the interiors of both of them. Indeed, otherwise we could slightly extend one or both of these segments around the meeting point.
Finally, we may assume that the segments of $\Ss$ lie on a grid of size $\lhor\times\lver$ so that the segments lying on the same grid line induce a path in the intersection graph. Indeed, if on a single grid line the segments induce a disjoint union of several paths, then we can move these paths slightly so that they are realized on separate grid lines. A representation $\Ss$ of $G$ with the properties described above is called \emph{favorable}. 


To give constructions for the subsets of pairwise disjoint segments in a favorable representation $\Ss$, we first need some notation.
Suppose the $\lhor\times\lver$ grid has $\leven$ grid lines with an even number of segments and $\lodd$ grid lines with an odd number of segments.
In total there are $\seven$ segments which lie on a grid line with an even number of segments and $\sodd$ segments which lie on a grid line with an odd number of segments.
The maximum number of segments lying on a single grid line is $t$.

The following three lemmas correspond each to a different set of pairwise disjoint segments in $\Ss$. In all three lemmas, we assume $\Ss$ to be a favorable representation of a graph in $\graphs$ with $n$ vertices.

\begin{lemma}\label{lem:odd technique}
	There exists a subset of $\Ss$ consisting of $\frac{n}{4} + \frac{\lodd}{4}$ pairwise disjoint segments.
\end{lemma}

\begin{lemma}\label{lem:line technique}
	There exists a subset of $\Ss$ consisting of $\frac{n}{4} + \frac{t}{4}$ pairwise disjoint segments.
\end{lemma}

\begin{lemma}\label{lem:even technique}
	There exists a subset of $\Ss$ consisting of $\frac{n}{4} + \frac{\sqrt{2\seven}}{4} - \frac{\lodd}{4}$ pairwise disjoint segments.
\end{lemma}

Before proving these lemmas, we use them to conclude \cref{thm:lower bound}.

\begin{proof}[Proof of \cref{thm:lower bound}]
	Let $G$ be a graph in the class $\graphs$ with $n$ vertices and let $\Ss$ be a favorable representation of $G$. A subset of pairwise disjoint segments in $\Ss$ corresponds to an independent set in $G$ of the same size. We distinguish three cases.
	If $\lodd\ge\sqrt{n}/c$ for some constant $c$, by \cref{lem:odd technique} $G$ has an independent set of size at least
	\[\frac{n}{4} + \frac{1}{4c}\cdot \sqrt{n}.\]
	If $\lodd\le\sqrt{n}/c$ and $\seven \ge 2n/c^2$, by \cref{lem:even technique} $G$ has an independent set of size at least
	\begin{align*}
	\frac{n}{4} + \frac{\sqrt{2\seven}}{4} - \frac{\lodd}{4}&\ge
	\frac{n}{4} + \frac{\sqrt{4n}}{4c} - \frac{\sqrt{n}}{4c}\\ &\ge
	\frac{n}{4} + \frac{1}{4c}\cdot \sqrt{n}.
	\end{align*}
	If $\lodd\le\sqrt{n}/c$ and $\seven \le 2n/c^2$, we get $\sodd \ge n(1-2/c^2)$ using $\seven+\sodd=n$. Then the maximum number of segments $t$ lying on a single line is at least
	\[
	t \ge \frac{\sodd}{\lodd} \ge \frac{n(1-2/c^2)}{\sqrt{n}/c} =
	\frac{c^2-2}{c}\cdot \sqrt{n}.
	\]
	By \cref{lem:line technique} we get an independent set of $G$ of size at least 
	\[\frac{n}{4} + \frac{c^2-2}{4c}\cdot \sqrt{n}.\]
	Setting $c=\sqrt{3}$ gives the desired result: there is always an independent set of size at least $\frac{n}{4}+\frac{1}{4\sqrt{3}}\cdot \sqrt{n}$.
\end{proof}

It remains to prove the three lemmas. 

\begin{proof}[Proof of \cref{lem:odd technique}]
	This construction exploits grid lines with an odd number of segments on them.
	For each grid line, select every second segment lying on that line, starting from the leftmost.
	If the grid line has an even number of segments, exactly half of the segments are selected.
	If the grid line has an odd number of segments, the selected number of segments is half rounded up.
	This corresponds to selecting exactly half of all the segments and adding $1/2$ for each grid line with an odd number of segments. In total,
	\[\frac{n}{2}+\frac{\lodd}{2}\]
	segments are selected.
	
	By construction of this subset, two segments are only intersecting if one is horizontal and the other one is vertical. The set can be partitioned into horizontal and vertical segments with both parts only containing pairwise disjoint segments. By the pigeonhole principle, one of the two parts contains at least half of the selected segments. 
\end{proof}

\begin{proof}[Proof of \cref{lem:line technique}]
	This construction exploits a single grid line with many segments on it.
	Let $g_\hor$ be a horizontal grid line with the maximum number of segments $t_\hor$ lying on it. Let $s_\ver$ be the total number of vertical segments.
	Let $\Ss_\hor$ be the set of consisting of all segments lying on $g_\hor$ and all vertical segments.
    Analogously define  $g_\ver$, $t_\ver$, $s_\hor$, and $\Ss_\ver$.
	Now we choose the larger set among $\Ss_\ver$ and $\Ss_\hor$.
	The size of this set is
	\begin{align*}
	\max\{s_\hor + t_\ver,s_\ver + t_\hor\} &\ge \frac{s_\hor + t_\ver + s_\ver + t_\hor}{2}\\
	&\ge \frac{n + t}{2}
	\end{align*}
	For
	the second inequality, we use assertions $s_\hor+s_\ver = n$ and $t=\max\{t_\hor,t_\ver\}$.
	
	We now observe that the intersection graphs of both sets $S_\ver$ and $S_\hor$ are bipartite.
	Indeed, any cycle in the intersection graph has to contain at least two horizontal segments lying on thwo different horizontal grid lines, and two vertical segments lying on two different vertical grid lines. But $S_\ver$ contains horizontal segments from only one horizontal grid line, while $S_\hor$ contains vertical segments from only one vertical grid line.
	In a bipartite graph the vertices can be partitioned into two independent sets $A$ and $B$, one of which contains at least half of the vertices. Hence, the larger of the sets $S_\ver$ and $S_\hor$ contains an independent set of size at least $\frac{n+t}{4}$.
\end{proof}

The proof of \cref{lem:even technique} heavily depends on the following classic theorem of Erd\H{o}s and Szekeres, here rephrased in the plane setting. We say that a sequence of points in the plane is {\em{non-decreasing}} if both their first and second coordinates are non-decreasing along the sequence; it is {\em{non-increasing}} if the first coordinate is non-decreasing along the sequence while the second is non-increasing.

\begin{theorem}[Erd\H{o}s, Szekeres \cite{erdos1935combinatorial}]\label{thm:erdos szekeres}
	Given $n$ distinct points on the plane, it is always possible to choose at least $\sqrt{n}$ of them and arrange into a sequence so that this sequence is either non-increasing or non-decreasing.
\end{theorem}

%

\begin{proof}[Proof of \cref{lem:even technique}]
	The construction exploits grid lines with an even number of segments. With the help of \cref{thm:erdos szekeres} we first construct a polyline that cuts through the segments. Then we use this polyline to define two sets of pairwise disjoint segments in $\Ss$, one of which has the desired size.
	
	Recall that meeting points are the points in which two parallel segments intersect. A meeting point on a grid line naturally partitions the segments lying on this line into two parts: those to the left of it and to the right of it (for horizontal lines), or those above it and below it (for vertical lines). Call a meeting point a \emph{candidate point} if both those parts have odd cardinalities. Note that thus, candidate points only occur on grid lines with an even number of segments. Further, in total there are $\seven/2$ candidate points. 
	
	By \cref{thm:erdos szekeres}, there exists either a non-increasing or a non-decreasing sequence of $\sqrt{\seven/2}$ candidate points.
	Suppose without loss of generality that the sequence is non-increasing and
    of maximum possible length. We call \emph{cutting points} the candidate points in the
    sequence and we use $C$  to denote the number of cutting points. Observe
    that $C \ge \sqrt{\seven/2}$.
	For every two consecutive cutting points, connect them with a segment. Then consider two half-lines with negative inclinations, one ending at the first cutting point and one starting at the last cutting point. This gives a polyline intersecting all vertical and horizontal grid lines. We call this path the \emph{cut}. 
	\begin{figure}[ht]
	    \centering
		\includegraphics[scale=0.7]{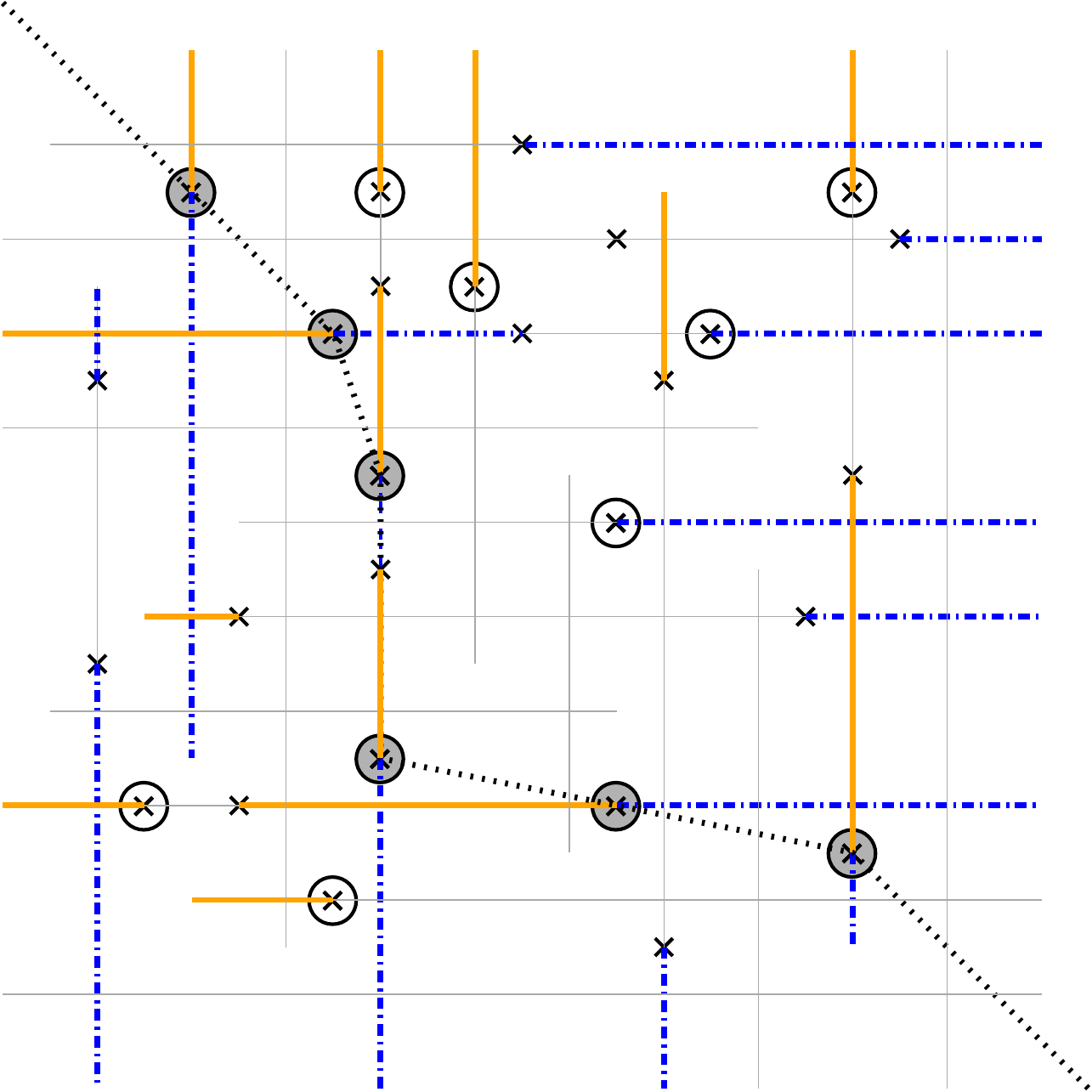}
		\caption{Selection of the two independent sets in the proof of \cref{lem:even technique}. Crosses are the meeting points, disks are the candidate points, and gray disks are the cut points. The black dotted line is the cut. Dashed blue and solid orange segments are those chosen to the sets $\Ss_\mathrm{blue}$ and $\Ss_\mathrm{orange}$, respectively.}
		\label{fig:even-technique}
	\end{figure}
	
	Using the cut, we construct two sets of segments $\Ss_\mathrm{blue}$ and $\Ss_\mathrm{orange}$; see \cref{fig:even-technique}. 
    
    \subparagraph*{Construction of $\Ss_\mathrm{blue}$:} The set $\Ss_\mathrm{blue}$ is constructed as follows.
    For each vertical grid line, start from the segment with the lowest endpoint
    and choose every second segment with the upper endpoint on the cut or
    below. Next, for each horizontal grid line, start from the segment with the
    right-most endpoint and choose every second segment with the left endpoint on the cut or to the right.
    
    \subparagraph*{Construction of $\Ss_\mathrm{orange}$:} The set $\Ss_\mathrm{orange}$ is symmetrically to $\Ss_\mathrm{blue}$. Namely, for each vertical grid
    line, start from the segment with the highest endpoint and choose every
    second segment with the lower endpoint on the cut or above.  For each
    horizontal grid line, start from the segment with the left-most endpoint and
    choose every second segment with the right endpoint on the cut or to
    the left.  
    
    If the sequence would be non-decreasing, the choice strategy for horizontal
    segments would be inverted between $\Ss_\mathrm{blue}$ and
    $\Ss_\mathrm{orange}$.
	
	We argue that the segments of $\Ss_\mathrm{blue}$ are pairwise disjoint. Note that the segments lying in the bottom-left side of the cut are vertical and pairwise disjoint by the construction, whole those lying in the top-right side of the cut are horizontal and pairwise disjoint. So it remains to argue that there is no pair of a vertical segment and a horizontal from $\Ss_\mathrm{blue}$ that would intersect at a point lying on the cut. Recall that since the representation is favorable, this intersection point would lie in the interiors of both segments. This would imply that either the vertical segment would have the top endpoint strictly above the cut, or the horizontal segment would have the left endpoint strictly to the left of the cut. This is a contradiction with the construction of $\Ss_\mathrm{blue}$. A symmetric argument shows that also the segments of $\Ss_\mathrm{orange}$ are pairwise~disjoint.
	
	It remains to show that $\Ss_\mathrm{blue} \cup \Ss_\mathrm{orange}$ has at least $\frac{n}{2} + \frac{\sqrt{2\seven}}{2} - \frac{\lodd}{2}$ segments.
	
	Consider a grid line with an even number of segments.
	For each candidate point on this line which is not a cutting point, exactly one segment containing this cutting point is in $\Ss_\mathrm{blue} \cup \Ss_\mathrm{orange}$. However, for each cutting point on this line, both segments meeting at this point are included in $\Ss_\mathrm{blue} \cup \Ss_\mathrm{orange}$, as there is an odd number of segments on either side. This means that on each such grid line, the total number of segments included in $\Ss_\mathrm{blue} \cup \Ss_\mathrm{orange}$ is exactly half of all the segments, plus one segment for each cutting point on the grid line.
	
	Consider now a grid line with an odd number of segments.
	The sets $\Ss_\mathrm{blue}$ and $\Ss_\mathrm{orange}$ contain every second segment starting from the outermost ones. Without the cut, this would include half of the segments lying on the line rounded up. Since there is an odd number of segments on the grid line, the cut crosses it only at one point. So at most one segment is removed from $\Ss_\mathrm{blue} \cup \Ss_\mathrm{orange}$ due to this. This means that among the segments lying on the line, at least half rounded down is included in $\Ss_\mathrm{blue} \cup \Ss_\mathrm{orange}$. This means we lose at most $1/2$ of a segment for each odd grid line.
	
	Together, this gives that $\Ss_\mathrm{blue} \cup \Ss_\mathrm{orange}$ contains at least
	\[\frac{\seven}{2} + C + \frac{\sodd}{2} - \frac{\lodd}{2} \ge \frac{n}{2} + \frac{\sqrt{2\seven}}{2} - \frac{\lodd}{2}\]
	segments.  By choosing the larger of the two sets, we obtain an independent set of the desired size.
\end{proof}

%% file: chapters/section1.tex
\section{The upper bound: proof of \cref{thm:upper_bound}} \label{sec:upper_bound}
In this section, we construct families of axis-parallel segments whose intersection graphs satisfying the requirements of \cref{thm:upper_bound}. In fact, we prove the following stronger statement.



\begin{theorem}[Full version of Theorem~\ref{thm:upper_bound}]\label{thm:ub-full}
    For any integer $k \geq 1$, there exists a graph $G_k$ in $\graphs$ on
    $4k^2$ vertices with clique covering number $\theta(G_k)=2k^2$, fractional independence number $\alpha^\star(G_k)=2k^2$, and
    independence number \[\alpha(G_k) = k^2 + 3k - 2.\]
\end{theorem}

Note that \Cref{cor:ratio,cor:LP} follow from \cref{thm:ub-full} by considering $G=G_k$ for $k$ large enough depending on $1/\eps$.
The remainder of this section is devoted to the proof of \cref{thm:ub-full}.

Fix an integer $k\geq 1$. We construct a set of $4k^2$ axis-parallel segments $\Mm_k$. The set $\Mm_k$ will consist of $k$ sets with $4k$ segments each; these sets will be called {\em{$k$-boxes}}.
A \emph{$k$-box} is a set of $4k$ axis-parallel segments distributed on $k$ horizontal and $k$ vertical lines, each with exactly two segments on it. For every line, the two segments on this line intersect at a single point, which we call their \emph{meeting point}.
In the construction of a $k$-box, the meeting points are arranged in a diagonal from the top left to the bottom right, see the case $k=6$ in \cref{fig:k_box}.
The \emph{up segments} (resp. \emph{down segments}) of a $k$-box are the segments lying vertically above (resp. below) a meeting point. Similarly, we define the \emph{left} and \emph{right segments} of a $k$-box.

\begin{figure}[ht]
\centering
	\includegraphics[scale =0.6]{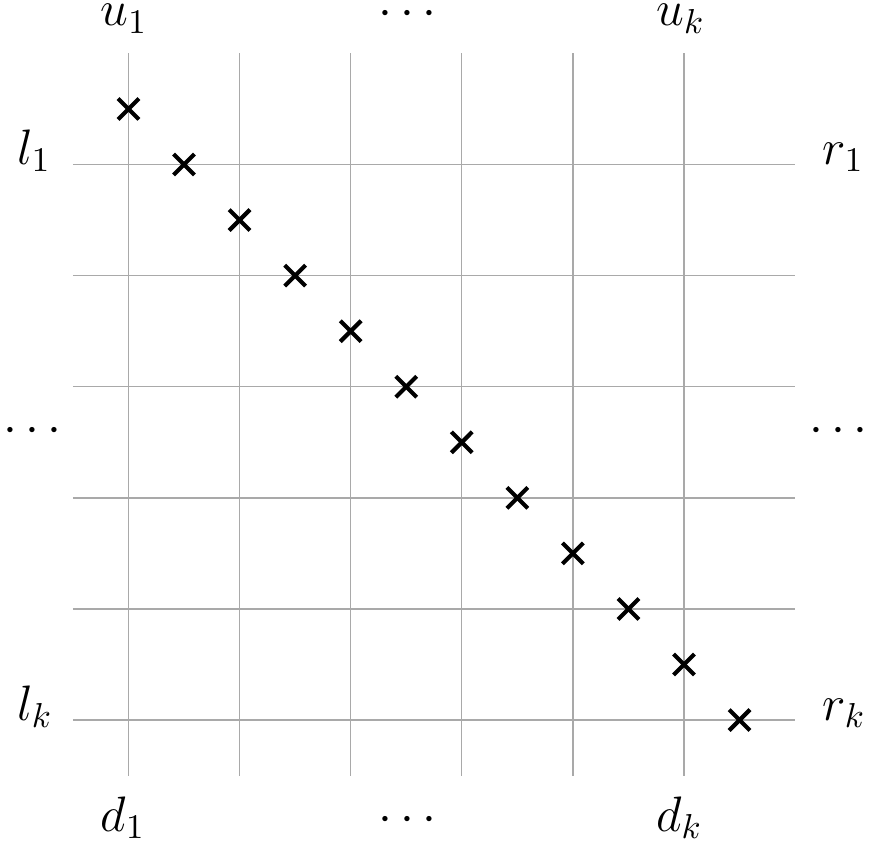}
	\caption{A 6-box. The meeting points are represented with crosses. Thus, every line contains two segments of the box, whose only intersection is the meeting point on this line.}
	\label{fig:k_box}
\end{figure}

To construct $\Mm_k$, consider a large square and place $k$ different $k$-boxes $\{\Bb_i\}_{i=1}^k$ along its diagonal from the bottom left to the top right. Then, prolong each segment away from the meeting point until it touches a side of the square, see \cref{fig:Michal_family}. The construction results in the set $\Mm_k$ consisting of $4k^2$ segments. We note that $\Mm_k$ is a favorable representation of its intersection graph in the sense introduced in \cref{sec:lower_bound}. Also, perhaps not surprisingly, the construction is inspired by a tight example for the Erd\H{o}s-Szekeres Theorem (\cref{thm:erdos szekeres}), so that it proves tightness of the bound provided by \cref{lem:even technique}. 


\begin{figure}[ht]
\centering
	\includegraphics[scale =0.6]{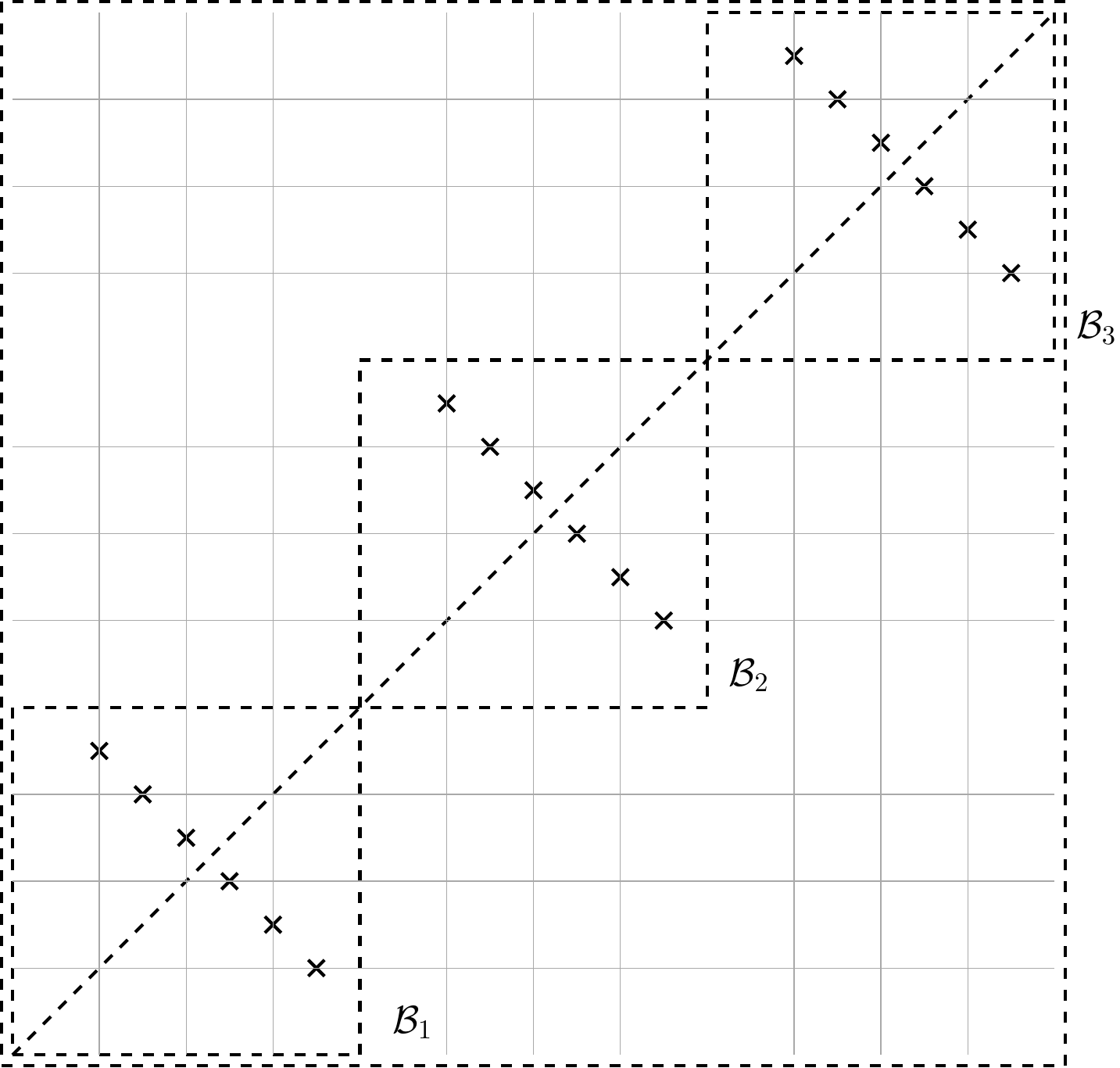}
	\caption{The set $\Mm_3$. The meeting points are represented by crosses. The dashed lines indicate the sides of the large square and of the $k$-boxes.}
	\label{fig:Michal_family}
\end{figure}

We are left with verifying the asserted properties of $\Mm_k$.
First, we introduce some notation and definitions.
Let $\Ii$ be a set of pairwise disjoint segments in $\Mm_k$.
A $k$-box $\Bb_i$ of $\Mm_k$ is said to be \emph{interesting} for $\Ii$ if $\Bb_i \cap \Ii$ contains  either at least one down segment and one right segment, or at least one up segment and one left segment. Otherwise, the $k$-box is \emph{boring} for $\Ii$.
Distinguishing between interesting and boring boxes allows more precise estimates on the maximum possible cardinality of $\Ii$.

In the next two lemmas, we consider $\Ii$ to be a set of pairwise disjoint segments in $\Mm_k$.
\begin{lemma}\label{lem:boring}
For any $k$-box $\Bb$ in $\Mm_k$,  $|\Bb \cap \Ii| \leq 2k$.
Moreover, if $\Bb$ is boring for $\Ii$, then $|\Bb \cap \Ii| \leq k+1$.
\end{lemma}

\begin{proof}
The first statement holds because $\Ii$ contains at most one segment per line, and there are $2k$ lines in a box: $k$ vertical and $k$ horizontal.

Assume now that $\Bb$ is a box that is boring for $\Ii$. Enumerate the up and down segments of $\Bb$ from left to right as $U = \{u_1, \ldots, u_k\}$ and $D = \{d_1, \ldots, d_k\}$, and the right and left segments from top to bottom as $R =\{r_1, \ldots, r_k\}$ and $L =\{l_1, \ldots, l_k\}$; see \cref{fig:k_box}.
If all segments of $\Bb\cap \Ii$ are pairwise parallel (that is, they are either all vertical or all horizontal), then $|\Bb \cap \Ii| \leq k$ since $\Ii$ can contain only one segment per line.
Then, there are two cases left to check: either $\Bb$ contains only up and right segments, or only down and left segments.
Observe that $U \cup R$ can be partitioned into $k+1$ parts as follows: $u_1$ and $r_k$ are in singleton parts, and we have $k-1$ pairs of intersecting segments $ \{u_{i+1},r_i\}_{i=1}^{k-1}$. Similarly,  $D \cup L$ can be  partitioned into $k$ pairs of intersecting segments $ \{d_i,l_i\}_{i=1}^k$.
The independent set $\Ii$ can contain at most one segment from each part of these partitions. Hence, $|\Bb \cap \Ii| \leq k+1$ in both cases.
\end{proof}

\begin{lemma}\label{lem:interesting}
There are at most two boxes that are interesting for $\Ii$.
\end{lemma}
\begin{proof}
We show that there is at most one interesting box with at least one up and one left segment included in $\Ii$. Then a symmetric argument shows that there is at most one interesting box with at least one down and one right segment included in $\Ii$, implying that there are at most two interesting boxes in total.

For the sake of contradiction, assume $\Mm_k$ there are two distinct interesting boxes $\Bb, \Bb'$ of the first kind.
Then, either an up segment of $\Bb \cap \Ii$ intersects a left segment of $\Bb' \cap \Ii$, or vice-versa.
This contradicts the fact that segments of $\Ii$ are pairwise disjoint.
\end{proof}

With the lemmas in place, we are in position to finish the proof of \cref{thm:ub-full}.
Let $G_k$ be the intersection graph of $\Mm_k$. By construction, the set $\Mm_k$ consists of $4k^2$ axis-parallel segments and $G_k$ is in $\graphs$.

First, we compute the clique covering number and the fractional independence number of $G_k$.
Observe that $G_k$ is triangle-free, hence every clique in $G_k$ is of size at most $2$. It follows that every clique covering of $G_k$ is of size at least $\frac{|\Mm_k|}{2}=2k^2$, that is, $\theta(G_k)\leq 2k^2$. On the other hand, taking every vertex of $G_k$ with multiplicity $1/2$ gives a fractional independent set of size $\frac{|\Mm_k|}{2}=2k^2$, implying that $\alpha^\star(G_k)\geq 2k^2$. Since $\theta(H)\geq \alpha^\star(H)$ for every triangle-free graph $H$, we conclude that 
$$\theta(G_k)=\alpha^\star(G_k)=2k^2.$$

It remains to prove that $\alpha(G_{k})= k^2 + 3k - 2$.
We give a set of pairwise disjoint segments in $\Mm_k$, corresponding to an independent set in $G_k$. This shows that $\alpha(G_k) \geq k^2 + 3k - 2$.
The set of segments in $\Mm_k$ consists of:
(i) the left and up segments of $\Bb_1$, and (ii) the right and down segments of
$\Bb_2$, and (iii) the right segments and the topmost up segment of $B_i$, for each $3 \leq i\leq k$.
This is a set of pairwise disjoint segments in $\Mm_k$ and it contains $2\cdot (2k) + (k-2)(k+1) = k^2 + 3k - 2$ segments.

To show that $\alpha(G_k) \leq k^2 + 3k - 2$ we apply \cref{lem:boring} and \cref{lem:interesting} to obtain, for any set $\Ii$ of pairwise disjoint segments in $\Mm_k$, that
$$ | \Ii | = |\Mm_k \cap \Ii | = \sum_{i=1}^k |\Bb_i \cap \Ii| \leq 2\cdot (2k) + (k-2)(k+1) = k^2 + 3k -2.$$
This concludes the proof of \cref{thm:ub-full}.

%% file: chapters/acknowledgement.tex
\paragraph*{Acknowledgements.} The results presented in this paper were obtained during the trimester on Discrete Optimization at Hausdorff Research Institute for Mathematics (HIM) in Bonn, Germany. We are thankful for the possibility of working in the stimulating and creative research environment at HIM.